\documentclass{amsart}
\usepackage{amssymb,latexsym}
\theoremstyle{plain}
\newtheorem{theorem}{Theorem}

\newtheorem{lemma}{Lemma}
\theoremstyle{definition}

\begin{document}

\title[maximal rank]
{Embeddings of general curves in projective spaces: the range of the quadrics}
\author{E. Ballico}
\address{Dept. of Mathematics\\
 University of Trento\\
38123 Povo (TN), Italy}
\email{ballico@science.unitn.it}
\thanks{The author was partially supported by MIUR and GNSAGA of INdAM (Italy).}

\begin{abstract}
Let $C \subset \mathbb  {P}^r$ a general embedding of prescribed degree of a general
smooth curve with prescribed genus. Here we prove that either $h^0(\mathbb {P}^r,\mathcal {I}_C(2)) =0$
or $h^1(\mathbb {P}^r,\mathcal {I}_C(2)) =0$ (a problem called the Maximal Rank Conjecture
in the range of quadrics).
\end{abstract}

\maketitle

\section{Introduction}\label{S1}

Let $C\subset \mathbb {P}^r$ be any projective curve. The curve $C$ is said to have {\it maximal rank} if for every
integer $x>0$ the restriction map $H^0(\mathbb {P}^r,\mathcal {O}_{\mathbb {P}^r}(x)) \to H^0(C,\mathcal
{O}_C(x))$ has maximal rank, i.e. either it is injective or it is surjective. 

For any curve $X$ and any spanned $L\in \mbox{Pic}(X)$ let $h_L: X \to \mathbb {P}^r$, $r:= h^0(X,L)-1$, denote the morphism induced by the complete
linear system $\vert L\vert$. Here we prove the following result, which improves one of the results in \cite{w}.

For all integers $g, r, d$ set $\rho (g,r,d):= (r+1)d-rg -r(r+1)$ (the Brill-Noether number for $g^r_d$'s on a curve of genus $g$). Fix integers $r\ge 3$, and $g \ge 3$. Fix a general $X\in \mathcal {M}_g$.
Brill-Noether theory says that $G^r_d(X)\ne \emptyset$ if and only if $\rho (g,r,d) \ge 0$ (equivalently, $W^r_d(X)\ne \emptyset$ if and only if $\rho (g,r,d)\ge 0$) (\cite{acgh}, Ch. V). The Maximal Rank Conjecture in $\mathbb {P}^r$
asks if a general embedding in $\mathbb {P}^r$ of a general curve has maximal rank. Since this is true for non-special embeddings (\cite{be1} if $r=3$, \cite{be3} if $r>3$), we only need to consider triples $(g, r, d)$ with $d<g+r$
and $\rho (g,r,d)\ge 0$. For these triples of integers Brill-Noether theory gives $W^r_d(X)\ne \emptyset$, that $W^r_d(X)$ has
pure dimension $\rho (g,r,d)$  (it is also irreducible
if $\rho (g,r,d) >0$) and that $W^r_d(X) \ne W^{r+1}_d(X)$, i.e. $h^0(X,L) = r+1$ for a general $L\in W^r_d(X)$ (\cite{acgh}, Ch. V) (in the case $\rho (g,r,d) =0$ we have $W^{r+1}_d(X)=\emptyset$). Hence $h^1(X,L)=r+1$
for a general $L\in W^r_d(X)$ (or for all $L\in W^r_d(X)$ if $\rho (g,r,d)=0$). For this range of triples $(g,r,d)$ it is very easy to prove that a general $L\in W^r_d(X)$ is very ample (e.g., see the proof of \cite{s0}, Theorem at pages 26-27). 

In this paper we answer a question raised in \cite{w} (J. Wang called it the Maximal Rank Conjecture for quadrics).

\begin{theorem}\label{i1}
Fix integers $g\ge 3$, $r\ge 2$ and $d$ such that $\rho (g,r,d)\ge 0$ and $d \le g+r$. Fix a general $L\in W^r_d(X)$. Then the symmetric multiplication rank $\mu _L: S^2(H^0(X,L)) \to H^0(X,L^{\otimes 2})$ has maximal rank, i.e.
it is either injective or surjective.
\end{theorem}

Since $\mu _L$ is obviously injective if $r=2$, to prove Theorem \ref{i1} we may assume $r\ge 3$. The surjectivity part in Theorem \ref{i1} is true (\cite{bf}, Theorem 1); this part corresponds to the triples $(g,r,d)$ with $r\ge 3$, $0\le g+r-d \le (r-1)/2$
and $(r+1)(g+r-d) \le g \le r(r-1)/2 +2(g+r-d)$.

Fix a general $X\in \mathcal {M}_g$. M. Teixidor i Bigas proved that if $d\le g+1$, then $\mu _L$ is injective for all $L\in \mbox{Pic}^d(X)$. See \cite{be2} and \cite{b} for the Maximal Rank Conjecture for $r\le 4$.

We work over an algebraically closed field with characteristic zero.

\section{Proof of Theorem \ref{i1}}\label{S2}

\begin{lemma}\label{e1}
Fix a set $S\subset \mathbb {P}^r$, $r\ge 3$, such that $\sharp (S) =r+2$ and $S$ is in linearly general position, i.e. any $r+1$ of its points span $\mathbb {P}^r$.
For each $P\in S$ let  $L_P \subset \mathbb {P}^r$ be a line such that $P\in L_P$. Fix any closed subset $T\subset \mathbb {P}^r$ such that $\dim (T) \le 1$. Then there exists
a rational normal curve $D \subset \mathbb {P}^r$ such that $S \subset D$, for each $P\in S$ the line $L_P$ is not the tangent line of $D$ at $P$ and $D\cap (T\setminus T\cap S)=\emptyset$.
\end{lemma}

\begin{proof}
Let $E(S)$ be the set of all rational normal curves $C\subset \mathbb {P}^r$ containing $S$. For each $Q\in (\mathbb {P}^r\setminus S)$ set
$E(S,Q):= \{C\in E(S): Q\in C\}$. If $S\cup \{Q\}$ is in linearly general position, then $\sharp (E(S,Q)) =1$. If $S\cup \{Q\}$ is not in general position, then $E(S,Q)=\emptyset$.
Hence $E(S)$ is a quasi-projective irreducible variety of dimension $r-1$. For each $P\in S$ and any tangent vector of $\mathbb {P}^r$ at $P$ set $E(\nu ):= \{C\in E(S): \nu$ is
tangent to $C$ at $P\}$. Either $E(\nu )=\emptyset$ or $E(\nu )$ is a single point (part (b) of \cite{eh}, Theorem 1). Since $\dim (\cup _{Q\in (T\setminus T\cap S)} E(S,Q)) \le \dim (T)=1$, a general $D\in E(S)$  satisfies the thesis of the lemma.
\end{proof}

\vspace{0.3cm}

\qquad {\emph {Proof of Theorem \ref{i1}.}} Fix integers $x, y$ such that $y \ge 0$ and either $x \ge y+r$ or $y\ge 2 $ and $x \ge r$ and $y \le x-r +\lfloor (x-r-2)/(r-2)\rfloor$. In the papers \cite{be2} (case $r=3$) and \cite{be4} (case $r\ge 4$) the authors defined an irreducible
component $W(x,y;r)$ of the Hilbert scheme of $\mathbb {P}^r$ whose general member is a non-degenerate smooth curve of genus $y$ and degree $x$. If $x\ge r+y$, this is the component of $\mbox{Hilb}(\mathbb {P}^r)$ whose general element is a non-special curve,
while if $x < y+r$, then $h^1(E,\mathcal {O}_E(1)) = y+r-x$ and $E$ is linearly normal. If $\rho (y,r,x) \ge 0$, then a general $E\in W(x,y;r)$ has general moduli (\cite{be4}, Proposition 3.1). Since $\mu _L$ is obviously injective if $r=2$, to prove Theorem \ref{i1} we may assume $r\ge 3$. Hence we may assume that $L$ is very ample (e.g., see the proof of \cite{s0}, Theorem at pages 26-27). Since $W^r_d(X)$ is irreducible, the semicontinuity theorem
for cohomology gives that to prove Theorem \ref{i1} for the triple $(g,r,d)$ it is sufficient to find one $L\in W^r_d(X)\setminus W^{r+1}_d(X)$ such that $\mu _L$ has maximal rank. Since $h^0(X,L)=r+1$, we have $\dim S^2(H^0(X,L)) = \binom{r+2}{2}$.
By Gieseker-Petri theory we have $h^1(X,L^{\otimes 2}) = 0$. Hence $h^0(X,L^{\otimes 2})=2d+1-g$. The surjectivity part in Theorem \ref{i1} is true (\cite{bf}, Theorem 1). Hence we may assume $2d+1-g > \binom{r+2}{2}$. Set $c:= g+r-d$. Riemann-Roch gives $c = h^1(X,L)$. Since
a general non-special embedding of $X$ has maximal rank (\cite{be1} for $r=3$, \cite{be2} for $r\ge 4$), we may assume $c>0$. Since
$\rho (d,g,r) = (r+1)(g+r-c)-rg -r(r+1)$, the assumption $\rho (g,r,d) \ge 0$ is equivalent to $g \ge c(r+1)$. For all integers $t \ge 3$ and $b \ge 0$ set $g_{t,b}:= t(t-1)/2 + 2b$. Notice that
\begin{equation}\label{eqe1}
2(g_{r,b}+r-b) +1 -g_{r,b} = \binom{r+2}{2}
\end{equation}

We have $\rho (g_{r,b},r,g_{r,b}+r-b) = g_{r,b} -b(r+1)$. Hence $\rho (g_{r,b},r,g_{r,b}+r-b)\ge 0$ if and only if $g_{r,b} \ge b(r+1)$. Hence $\rho (g_{r,b},r,g_{r,b}+r-b)\ge 0$ if and only if $b \le r/2$.

\quad (a) Here we assume $c \le r/2$. Hence $\rho (g_{r,c},r,g_{r,c}+r-c) \ge 0$. Fix a general $E\in \mathcal {M}_g$ and a general $R\in W^r_{g_{r,c}+r-c}(E)$. Brill-Noether theory gives $h^0(E,R)=r+1$. By \cite{bf}, Theorem 1, $R$ is
very ample and the curve $h_R(E)$ is projectively normal. Hence $\mu _R$ is surjective. The case $b = c$ of (\ref{eqe1}) gives that $\mu _R$ is bijective. Hence $h^0(\mathbb {P}^r,\mathcal {I}_{h_R(E)}(2)) =0$. Since
$E$ has general moduli, we have $h_R(E)\in W(g_{r,c}+r-c,g_{r,c};r)$. Since $2(g+r-c)+1-g > \binom{r+2}{2}$,
we have $g > g_{r,c}$. We have $d = g+r-c$ and hence $g -g_{r,c} = d -(g_{r,c}+r-c)$. Let $A\subset \mathbb {P}^r$ be the union of $h_R(E)$ and $g-g_{r,c}$ general secant lines of $h_R(E)$.
We have $A\in W(d,g;r)$ (apply $c$ times \cite{be4}, Lemma 2.2). Since $A\supset E$ and $h^0(\mathbb {P}^r,\mathcal {I}_{h_R(E)}(2)) =0$, we have $h^0(\mathbb {P}^r,\mathcal {I}_A(2)) =0$. Hence $h^0(\mathbb {P}^r,\mathcal {I}_C(2))=0$
for a general $C\in W(d,g;r)$. Since $\rho (d,g,r)\ge 0$, $C$ is a linearly normal curve of degree $d$ and genus $g$ with general moduli. By semicontinuity we get the injectivity of $\mu _L$ for a general $X$ and a general $L$.

\quad (b). From now on (i.e. in steps (b), (c), (d), (e)) we assume $c > r/2$. In this step we assume $r$ even.  Notice that $g_{r,r/2} = r(r+1)/2$. Hence $\rho (g_{r,r/2},g_{r,r/2} +r -r/2 ,r) =0$, $W(g_{r,r/2}+r/2,g_{r,r/2};r)$
is defined and a general element of it has general moduli. Fix a general $Y\in W(g_{r,r/2}+r/2,g_{r,r/2};r)$. Since $Y$ has general moduli, \cite{bf}, Theorem 1, and (\ref{eqe1}) give $h^i(\mathbb {P}^r,\mathcal {I}_Y(2))=0$, $i=0,1$.
Set $k:= g -c(r+1)$. Fix a general $S\subset Y$ such that $\sharp (S) = (r+2)(c-r/2)$ and take a partition of $S$ into $c-r/2$ disjoint sets $S_i$, $1 \le i \le c-r/2$, such that $\sharp (S_i) = r+2$ for all $i$. Let $E\subset \mathbb {P}^r$ be a general union of $Y$, $c-r/2$ rational normal curves $E_i$, $1 \le i \le c-r/2$, such that $S_i  \subset  D_i$ for all $i$ and $k$ general secant lines $R_j$, $1\le j \le k$. We may find these rational normal curves and these lines so that $D_i\cap R_j=\emptyset$ for all $i, j$, $S_i=D_i\cap Y$ for all $i$, each $D_i$
intersects quasi-transversally $Y$, $D_i\cap D_h=\emptyset$ for all $i\ne h$, $R_j\cap R_k=\emptyset$ for all $i\ne j$, each $R_j$ intersects $Y$ quasi-transversally and $\sharp (R_i\cap Y) =2$ for all $i$ (first add the $k$ general secant lines
and then use Lemma \ref{e1}). Notice that $E$ is a nodal curve of degree $d$ and arithmetic genus $g$. Since $E\supseteq Y$, we have $h^0(\mathbb {P}^r,\mathcal {I}_E(2))=0$. By \cite{be4}, Lemmas 2.2 and 2.3, we have $E\in W(d,g;r)$. By semicontinuity
we have $h^0(\mathbb {P}^r,\mathcal {I}_F(2)) =0$ for a general $F\in W(d,g;r)$. Since $\rho (g,r,d) \ge 0$, $F$ has general moduli (\cite{be4}, Proposition 3.1).

\quad (c) From now on we assume $r$ odd. Since the case $r=3$ is true (e.g. by \cite{w}, Theorem 1.6, or by \cite{be2}, Theorem 1), we assume $r\ge 5$.

\quad  (d) In this step we prove the existence of $Y_2\in W((r^2+2r+1)/2,(r^2+r+2)/2;r)$ such that $h^i(\mathbb {P}^r,\mathcal {I}_{Y_2}(2))=0$, $i=0,1$. Notice that 
$g_{r,(r+1)/2} = r(r-1)/2 +r+1 = (r^2+r+2)/2$ and $g_{r,(r+1)/2}+r-(r+1)/2 = (r^2+2r+1)/2$. Hence $2((r^2+2r+1)/2) +1-(r^2+r+2)/2) = \binom{r+2}{2}$. The irreducible component $W((r^2+2r+1)/2,(r^2+r+2)/2;r)$ of $\mbox{Hilb}(\mathbb {P}^r)$ is defined, because $(r^2+2r+1)/2 \ge r$, $(r^2+r+2)/2 \ge 2$
and $(r^2+r+2)/2 \le (r^2+1)/2 + \lfloor (r^2-3)/(2r -4)\rfloor$ (the latter inequality is equivalent to the inequality $(r+1)(r-2) \le r^2-3$).
Notice that $\rho ((r^2+r+2)/2,r,(r^2+2r+1)/2)<0$ and hence this case corresponds
to a case with $d'-g' = c = (r+1)/2$, but in a range of triples $(d',g',r)$ for which there is no curve with general moduli. Fix a hyperplane $H\subset \mathbb {P}^r$. By \cite{bf}, Theorem 1, applied in $H $ there is a smooth curve $Y_1\subset H$ such that $Y_1\in W(g_{r-1,(r-1)/2}+(r-1)/2,g_{r-1,(r-1)/2};r-1)$
and $h^i(H,\mathcal {I}_{Y_1}(2)) =0$, $i=0,1$. We have $g_{r-1,(r-1)/2} = (r-1)r/2$, $g_{r-1,(r-1)/2}+(r-1)/2 = (r^2-1)/2$ and $\rho (g_{r-1,(r-1)/2},r-1,g_{r-1,(r-1)/2}+(r-1)/2) =0$. Fix a general
$S\subset Y_1$ such that $\sharp (S) =r+1$. Let $B\subset \mathbb {P}^r$ be a smooth and linearly normal elliptic curve such that $B\cap H = S$ ($B$ exists, because any two subsets of $H$ with cardinality $r+1$ and
in linearly general position are projectively equivalent). Set $Y_2:= Y_1\cup B$. The curve $Y_2$ is a connected and nodal curve with degree $(r^2+2r+1)/2$ and arithmetic genus $(r^2+r+2)/2$.  By \cite{bf}, Lemma 7, we have $Y_2 \in W((r^2+2r+1)/2,(r^2+r+2)/2;r)$.
Since $B$ is an elliptic curve of degree $\sharp (S)$, we have $h^0(B,\mathcal {O}_B(2)(-S)) =0$. Hence the Mayer-Vietoris exact sequence
$$0 \to \mathcal {O}_{Y_2}(2) \to \mathcal {O}_{Y_1}(2)\oplus \mathcal {O}_B(2) \to \mathcal {O}_S(2) \to 0$$
gives $h^1(Y_2,\mathcal {O}_{Y_2}(2))=0$. Since $2\cdot \deg (Y_2)+1-p_a(Y_2) = \binom{r+2}{2}$, we have $h^1(\mathbb {P}^r,\mathcal {I}_{Y_2}(2)) = h^0(\mathbb {P}^r,\mathcal {I}_{Y_2}(2)) $.
Fix $f\in H^0(\mathbb {P}^r,\mathcal {I}_{Y_2}(2)) $. Since $f\vert H$ vanishes on $Y_1$ and $h^0(H,\mathcal {I}_{Y_1}(2))=0$, $f$ is divided by the equation $z$ of $H$. Hence $f/z\in H^0(\mathbb {P}^r,\mathcal {I}_B(2))$. Since $B$ spans $\mathbb {P}^r$, we
get $f/z =0$, i.e. $f=0$. Hence $h^i(\mathbb {P}^r,\mathcal {I}_{Y_2}(2))=0$, $i=0,1$.

\quad (e) Here we assume $r\ge 5$, $r$ odd, $c = (r+1)/2$ and $g = (r^2+2r+1)/2$. Hence $d = (r+1)^2/2 + (r-1)/2 = (r^2+3r)/2$. Let $Y_3$ be a general union of $Y_2$ and $(r-1)/2$ secant lines of $Y_3$. Since
$h^0(\mathbb {P}^r,\mathcal {I}_{Y_2}(2))=0$, we have $h^0(\mathbb {P}^r,\mathcal {I}_{Y_3}(2))=0$. Since $Y_2\in W((r^2+2r+1)/2,(r^2+r+2)/2;r)$, we have $Y_3\in W((r^2+3r)/2,(r^2+2r+1)/2;r)$ (\cite{be4}, Lemma 2.2). By
semicontinuity we have $h^0(\mathbb {P}^r,\mathcal {I}_{Y_4}(2))=0$ for a general $Y_4\in W((r^2+3r)/2,(r^2+2r+1)/2;r)$. Since $\rho ((r^2+2r+1)/2,r,(r^2+3r)/2)=0$, $Y_4$ has general moduli.

\quad (f) In this step we assume $r$ odd, $r \ge 5$ and $(c,g) \ne ((r+1)/2,(r+1)^2/2)$, i.e. in this step we prove all the cases not yet proven. As in step (b) set
$k:= g-c(r+1)$. Fix a general $S\subset Y$ such that $\sharp (S) = (r+2)(c-r/2)$ and take a partition of $S$ into $c-r/2$ disjoint sets $S_i$, $1 \le i \le c-(r+1)/2$, such that $\sharp (S_i) = r+2$ for all $i$. Let $E_4\subset \mathbb {P}^r$ be a general union of $Y_4$, $c-(r+1)/2$ rational normal curves $D_i$, $1 \le i \le c-(r+1)/2$, such that $S_i  \subset  D_i$ for all $i$ and $k$ general secan lines $R_j$, $1\le j \le k$. We may find these rational normal curves and these lines so that $D_i\cap R_j=\emptyset$ for all $i, j$, $S_i=D_i\cap Y$ for all $i$, each $D_i$
intersects quasi-transversally $Y_4$, $D_i\cap D_h=\emptyset$ for all $i\ne h$, $R_j\cap R_k=\emptyset$ for all $i\ne j$, each $R_j$ intersects $Y$ quasi-transversally and $\sharp (R_i\cap Y) =2$ for all $i$. Since
$h^0(\mathbb {P}^r,\mathcal {I}_{E_4}(2)) \le h^0(\mathbb {P}^r,\mathcal {I}_{Y_4}(2))=0$, we conclude as in step (b).\qed

\providecommand{\bysame}{\leavevmode\hbox to3em{\hrulefill}\thinspace}

\end{document}